\pdfoutput=1
\documentclass[letterpaper]{article}

\usepackage{geometry}
\usepackage[T1]{fontenc}
\usepackage[utf8]{inputenc}
\usepackage{newtxtext}
\usepackage{amsmath,amsthm}
\usepackage{newtxmath}
\usepackage[numbers]{natbib}
\usepackage{booktabs}
\usepackage{tikz}
\usepackage{listings}
\usepackage{xcolor}
\usepackage[scaled=0.85]{sourcecodepro}
\usepackage{xspace}
\usepackage[hyphens]{url}
\usepackage[hidelinks]{hyperref}
\hypersetup{%
  pdftitle={Biplanar graphs with independence number two are 9-colorable},%
  pdfauthor={Stefan Szeider},%
  pdfkeywords={biplanar graphs, thickness, Earth-Moon problem, chromatic number, SAT, Lean}}

\date{}

\lstdefinelanguage{lean}{
  keywords={theorem,lemma,def,structure,where,fun,by,exact,namespace,end,open,import,variable,Prop,Type},
  sensitive=true,
  comment=[l]{--},
  morecomment=[s]{/-}{-/},
}

\newtheorem{theorem}{Theorem}
\newtheorem{lemma}[theorem]{Lemma}
\newtheorem{corollary}[theorem]{Corollary}

\theoremstyle{definition}

\theoremstyle{remark}

\newcommand{\SB}{\{\,}
\newcommand{\SM}{\;{:}\;}
\newcommand{\SE}{\,\}}
\newcommand{\hy}{\hbox{-}\nobreak\hskip0pt}
\newcommand{\Lean}{Lean\xspace}
\newcommand{\CaDiCaL}{CaDiCaL\xspace}
\newcommand{\cnf}{\textsc{cnf}\xspace}
\DeclareMathOperator{\ee}{e}
\newcommand{\match}{\nu}

\title{Biplanar graphs with independence number two\\ are 9-colorable}
\author{Stefan Szeider\\[4pt]
  \small Algorithms and Complexity Group\\[-2pt]
  \small TU Wien, Vienna, Austria\\[-2pt]
  \small \href{https://ac.tuwien.ac.at/stefan-szeider/}{https://ac.tuwien.ac.at/stefan-szeider/}}

\begin{document}
\maketitle
\thispagestyle{empty}

\begin{abstract}
A graph is biplanar if it is the union of two planar graphs on the same vertex set.
The largest chromatic number of a biplanar graph is known to lie between 9 and 12.
The lower bound comes from Sulanke's graph, which has independence number 2, and a biplanar graph on 19 vertices with independence number 2 would have chromatic number at least 10.
Gethner and Sulanke asked in 2009 whether such a graph exists.
We show that it does not, and more generally that every biplanar graph with independence number at most 2 is 9-colorable.
The proof embeds a hypothetical counterexample in the union of two sphere triangulations, enumerates with SAT modulo symmetries the 3271 graphs that pass a necessary filter for the complement of such a union, and shows with a SAT solver that none of them is such a complement; a matching argument reduces the general statement to this computation and one further case on 18 vertices.
The computational part of the proof, including the completeness of the enumeration and every refutation, is checked in \Lean~4, assuming three classical facts about planar graphs.
The \Lean development, the SAT instances, and the enumeration certificates are available on Zenodo.
\end{abstract}

\section{Introduction}
\label{sec:intro}

A graph is \emph{biplanar} if its edge set is the union of the edge sets of two planar graphs on the same vertex set (i.e., if its thickness is at most two).
Ringel~\citep{ringel1959,jacksonringel2000} asked for the largest chromatic number of a biplanar graph; in the map formulation that gave the problem its name, every country on Earth owns a colony on the Moon, and a country and its colony receive the same color~\citep{hutchinson1993,gethner2018}.
A biplanar graph on $n\ge 3$ vertices has at most $6n-12$ edges~\citep{mutzelodenthalscharbrodt1998} and hence a vertex of degree at most 11; the same holds for every subgraph, so biplanar graphs are 11-degenerate and 12 colors always suffice.
Sulanke found in 1973 that the join of a 5-cycle and a complete graph on 6 vertices is biplanar; it has chromatic number~9, and it appeared in Gardner's column of 1980~\citep{gardner1980,gethnersulanke2009}.
Since then the answer has been known to lie in $\{9,10,11,12\}$.
Gethner~\citep{gethner2018} conjectures that the answer is 11, and it is not even known whether 11 colors suffice~\citep[Problem~2]{barbados2018}.

The lower bound rests on graphs with small independence number.
If the complement of a graph~$G$ on $n$ vertices contains no complete graph on $m\ge 2$ vertices, then every color class of $G$ has fewer than $m$ vertices and $\chi(G)\ge\lceil n/(m-1)\rceil$~\citep{gethnersulanke2009}.
Sulanke's graph has independence number 2, and so do the eight 9-critical biplanar graphs on 17 vertices with triangle-free complement found by Gethner and Sulanke~\citep{gethnersulanke2009} and the thirty 17-vertex examples among the 90 further 9-critical graphs found by Alalqam~\citep{alalqam2012}.
Further 9-critical biplanar graphs on 12 to 16 vertices are given by Boutin, Gethner, and Sulanke~\citep{boutingethnersulanke2008}.
For independence number 2 the bound reads $\chi(G)\ge\lceil n/2\rceil$, which first exceeds 9 at $n=19$.
Gethner and Sulanke~\citep[Open Problem~1(a)]{gethnersulanke2009} therefore posed the problem: ``Find a thickness-two graph on 19 vertices with triangle-free complement.''
The candidate named by Gethner~\citep{gethner2018} is $C_5[4,4,4,4,3]$, the 5-cycle with its vertices replaced by cliques of sizes $4,4,4,4,3$, where the cliques of two adjacent cycle vertices are joined by all edges; the notation is that of Albertson, Boutin, and Gethner~\citep{albertsonboutingethner2010}.
It was shown not to be biplanar by a SAT computation~\citep{kirchwegerscheucherszeider2023} and by an edge count that is part of a proof that every biplanar clique blow-up of the 5-cycle is 9-colorable~\citep{trivedi2026}.
Heuristic searches aimed directly at 19 vertices and triangle-free complement produced no such graph~\citep{alalqam2012,flower2017,weaver2023}.

A direct attack on the 19-vertex question is out of reach.
Deciding whether a graph is biplanar is NP-complete~\citep{mansfield1983}; biplanar graphs can contain expanders~\citep{dsw2016}, so biplanarity alone yields no planar-style separator theorem~\citep{hendreywood2019}; and there are far too many graphs on 19 vertices with independence number 2 to test them one by one.
Our proof reduces them to 3271 candidates for the complement of a union of two triangulations and refutes each candidate with a certificate that can be checked.

\begin{theorem}
\label{thm:main}
Every biplanar graph with independence number at most $2$ is 9-colorable.
In particular, every biplanar graph with independence number at most $2$ has at most $18$ vertices.
\end{theorem}

The main computational result is the following theorem, which answers Open Problem~1(a) of Gethner and Sulanke in the negative.

\begin{theorem}
\label{thm:nineteen}
No biplanar graph on $19$ vertices has independence number at most $2$.
\end{theorem}

Theorem~\ref{thm:main} concerns graphs of every order, but only finitely many cases need a computation.
If~$G$ is a biplanar graph on $n$ vertices with $\alpha(G)\le 2$, then $\overline{G}$ is triangle-free, so the neighborhood of each vertex in $\overline{G}$ is a clique of~$G$; since $K_9$ is not biplanar, $\overline{G}$ has maximum degree at most 8 and hence at most $4n$ edges, while $G$ has at most $6n-12$, and together these bounds force $n\le 19$.
Theorem~\ref{thm:nineteen} excludes $n=19$, and for $n\le 18$ a minimal counterexample to 9-colorability reduces, by a matching argument in the complement and an arithmetic case analysis, to an 18-vertex graph with a vertex adjacent to all others, which the next lemma excludes.

\begin{lemma}
\label{lem:apex}
No biplanar graph on $18$ vertices with a vertex adjacent to all other vertices has independence number at most $2$.
\end{lemma}

\begin{corollary}
\label{cor:alpha}
Every biplanar graph on at least $19$ vertices has an independent set of size $3$, and every biplanar graph with chromatic number at least $10$ has independence number at least $3$.
\end{corollary}

Theorem~\ref{thm:nineteen} and Lemma~\ref{lem:apex} are proved by computer, and the proofs are checked in the proof assistant \Lean~4~\citep{demoura2021}.
Theorem~\ref{thm:main} follows from them by a short argument on paper (Section~\ref{sec:main}), whose arithmetic part is again checked in \Lean.

\paragraph{The proof idea.}
If $G$ is a biplanar graph on 19 vertices with independence number at most 2, then each of its two planar layers extends to a triangulation of the sphere on the same 19 vertices, and the union $U$ of the two triangulations is again a biplanar graph with independence number at most 2.
The complement $H$ of $U$ is triangle-free and has no independent set of size 9, its maximum degree is at most 8 because $K_9$ is not biplanar, and counting edges shows that it has between 69 and 76 edges (Section~\ref{sec:reduction}).
Up to isomorphism, 3271 graphs satisfy these necessary conditions; we enumerate them with SAT modulo symmetries and certify in \Lean that every graph satisfying the conditions is isomorphic to one on the list (Section~\ref{sec:enum}).
The enlargement to $U$ turns the condition ``biplanar,'' whose direct encoding would have to enforce planarity of two layers, into the condition that certain triangles are the faces of two triangulations.
For each $H$ we formulate the latter as a propositional formula and refute it with a SAT solver (Section~\ref{sec:exact}).

\paragraph{What is verified.}
Every object that enters the proof of Theorem~\ref{thm:nineteen} and Lemma~\ref{lem:apex} is defined in \Lean: the reduction, the filter that describes the candidate complements, the propositional encoding of the face condition, and the list of candidates.
The SAT solver's refutations are imported into \Lean as LRAT proofs (a clausal proof format with hints that a checker can verify independently of the solver~\citep{cruzfilipe2017}) and checked by \Lean's own verified checker.
The final theorems depend on three classical facts about planar graphs, stated as hypotheses~(X1) to~(X3) in Section~\ref{sec:reduction}; planarity itself is not formalized.
Table~\ref{tab:trust} in Section~\ref{sec:lean} lists what is verified and what is trusted.
The \Lean development, the instances and the enumeration certificates are available on Zenodo~\citep{zenodo}.

\paragraph{Related verification work.}
Resolving a combinatorial question by SAT and certifying the result has precedents: Keller's conjecture~\citep{brakensiek2022}, the empty hexagon number~\citep{heulescheucher2024}, and the packing chromatic number of the infinite square grid~\citep{subercaseauxheule2023}.
The reduction for Keller's conjecture was later formalized~\citep{clune2023}, Gallicchio et al.~\citep{gallicchio2026} extended the formalization to the SAT encoding and the symmetry breaking, and Subercaseaux et al.~\citep{subercaseaux2024} formalized the empty hexagon result.
Our development follows the template of these two \Lean formalizations, which verify their encodings, and differs as follows.
Subercaseaux et al.\ pass their refutations through the verified external checker cake\_lpr~\citep{tanheulemyreen2023} and assert unsatisfiability of the formula as an axiom in \Lean.
Gallicchio et al.\ check the refutations for dimensions up to~5 inside \Lean, by compiled evaluation as we do; for dimensions 6 and~7, the solver's proofs and the symmetry-breaking proofs are checked by separate verified checkers, and in their published artifact these verdicts do not enter the \Lean theorem.
Unlike both, every refutation enters the final \Lean theorem: the LRAT proofs of the 3275 refutations in Section~\ref{sec:exact} are streamed from the solver into \Lean with lrat-catcher~\citep{lratcatcher}, so no proof file is stored.
The completeness of the enumeration, which rests on symmetry breaking, is verified in \Lean as well.
For this we extend LeanSMS~\citep{kirchwegermanriqueszeider2026}, which verifies non-existence results obtained with SAT modulo symmetries and names certified enumeration up to isomorphism as future work; earlier SMS work certified enumerations outside \Lean~\citep{kirchwegerszeider2024}.

\section{The reduction}
\label{sec:reduction}

All graphs are finite and simple.
For a graph $G$ we write $\overline{G}$ for its complement, $\ee(G)$ for its number of edges, $\deg_G(v)$ for the degree of a vertex $v$, $\alpha(G)$ for its independence number, $\chi(G)$ for its chromatic number, and $\match(G)$ for the size of a largest matching.
The \emph{join} $A\vee B$ of two graphs is obtained from disjoint copies of $A$ and $B$ by adding all edges between them.
In this notation, $G$ on vertex set $V$ is biplanar if there are planar graphs $P_1$ and $P_2$ on $V$ with $E(G)\subseteq E(P_1)\cup E(P_2)$; by~(X1) below we may assume equality.
A graph $G$ has $\alpha(G)\le 2$ if and only if $\overline{G}$ is triangle-free.

We use three facts about planar graphs.
\begin{enumerate}
\item[(X1)] Subgraphs and induced subgraphs of planar graphs are planar.
\item[(X2)] A planar graph on $n\ge 4$ vertices is a spanning subgraph of a planar sphere triangulation on the same vertex set.
Here a \emph{sphere triangulation} is a connected graph $T$ on $n$ vertices with $3n-6$ edges together with a set of \emph{faces}, which are triangles of $T$, such that every edge of $T$ lies in exactly two faces and, for every vertex $v$, the \emph{link} of $v$ is a single cycle; the link of $v$ is the graph on the neighbors of $v$ in which $a$ and $b$ are adjacent if $\{v,a,b\}$ is a face.
\item[(X3)] $K_9$ is not biplanar~\citep{bhk1962,tutte1963,biniaz2022}.
\end{enumerate}
A sphere triangulation is a combinatorial object, and planarity of $T$ enters the formal proof only through~(X2).
The combinatorial conditions do force a sphere: counting incidences gives $2n-4$ faces, the faces form a connected closed surface because every link is a single cycle, and its Euler characteristic is $n-(3n-6)+(2n-4)=2$.
Since this is not formalized, biplanarity is carried as a separate hypothesis wherever it is needed.
A triangle of $T$ need not be a face (it may be a separating triangle).

\begin{lemma}[enlargement]
\label{lem:enlarge}
Let $G$ be a biplanar graph on $n\ge 4$ vertices with $\alpha(G)\le 2$.
Then there are sphere triangulations $T_1$ and $T_2$ on $V(G)$ such that $U=T_1\cup T_2$ satisfies $G\subseteq U$, $\alpha(U)\le 2$, and~$U$ is biplanar.
\end{lemma}

\begin{proof}
Extend the two planar layers of $G$ by~(X2).
Adding edges does not increase the independence number.
\end{proof}

We call a graph $U$ an \emph{exact union} if $U=T_1\cup T_2$ for two sphere triangulations on $V(U)$.
Lemma~\ref{lem:enlarge} uses the completion of a layer to a triangulation, which Tutte performs by transferring edges from the other layer in his proof that $K_9$ is not biplanar~\citep{tutte1963} and which Battle, Harary, and Kodama obtain from F\'ary's theorem~\citep{bhk1962}; here both layers are completed, and they may share edges.
Figure~\ref{fig:exactunion} shows Sulanke's graph, which is itself an exact union.

\begin{figure}[t]
\centering
\begin{tikzpicture}[scale=0.8041,
  edge/.style={line width=0.4pt,gray!65},
  hedge/.style={line width=0.6pt,black},
  dbl/.style={line width=1.25pt,black},
  shade/.style={fill=gray!22},
  vtx/.style={circle,minimum size=0.41cm,inner sep=0pt,font=\scriptsize,line width=0.5pt,transform shape},
  cyc/.style={vtx,fill=black!55,text=white,draw=black!55},
  clq/.style={vtx,fill=white,draw=black},
  ttl/.style={font=\normalsize,anchor=base}]
\begin{scope}[xshift=0.000cm]
\fill[shade] (-0.129,-1.543) -- (0.322,-0.595) -- (2.520,-2.310) -- cycle;
\draw[edge] (0.086,1.023) -- (-0.000,2.310);
\draw[dbl] (0.086,1.023) -- (-0.640,0.257);
\draw[edge] (0.086,1.023) -- (0.147,0.175);
\draw[dbl] (0.086,1.023) -- (2.520,-2.310);
\draw[edge] (-0.000,2.310) -- (-0.640,0.257);
\draw[edge] (-0.000,2.310) -- (2.520,-2.310);
\draw[edge] (-0.000,2.310) -- (-2.520,-2.310);
\draw[edge] (-0.834,-1.867) -- (-0.129,-1.543);
\draw[edge] (-0.834,-1.867) -- (-0.640,0.257);
\draw[edge] (-0.834,-1.867) -- (-1.186,-1.183);
\draw[edge] (-0.834,-1.867) -- (2.520,-2.310);
\draw[edge] (-0.834,-1.867) -- (-2.520,-2.310);
\draw[edge] (-0.129,-1.543) -- (0.322,-0.595);
\draw[edge] (-0.129,-1.543) -- (-0.640,0.257);
\draw[edge] (-0.129,-1.543) -- (0.611,-1.335);
\draw[edge] (-0.129,-1.543) -- (2.520,-2.310);
\draw[edge] (0.322,-0.595) -- (-0.640,0.257);
\draw[edge] (0.322,-0.595) -- (0.611,-1.335);
\draw[edge] (0.322,-0.595) -- (0.147,0.175);
\draw[edge] (0.322,-0.595) -- (2.520,-2.310);
\draw[edge] (-0.640,0.257) -- (0.147,0.175);
\draw[edge] (-0.640,0.257) -- (-1.186,-1.183);
\draw[edge] (-0.640,0.257) -- (-2.520,-2.310);
\draw[dbl] (0.611,-1.335) -- (2.520,-2.310);
\draw[edge] (0.147,0.175) -- (2.520,-2.310);
\draw[dbl] (-1.186,-1.183) -- (-2.520,-2.310);
\draw[edge] (2.520,-2.310) -- (-2.520,-2.310);
\node[cyc] at (0.086,1.023) {$c_1$};
\node[cyc] at (-0.000,2.310) {$c_2$};
\node[cyc] at (-0.834,-1.867) {$c_3$};
\node[cyc] at (-0.129,-1.543) {$c_4$};
\node[cyc] at (0.322,-0.595) {$c_5$};
\node[clq] at (-0.640,0.257) {$k_1$};
\node[clq] at (0.611,-1.335) {$k_2$};
\node[clq] at (0.147,0.175) {$k_3$};
\node[clq] at (-1.186,-1.183) {$k_4$};
\node[clq] at (2.520,-2.310) {$k_5$};
\node[clq] at (-2.520,-2.310) {$k_6$};
\node[ttl] at (0,-3.29) {$T_1$};
\end{scope}
\begin{scope}[xshift=6.080cm]
\draw[edge] (0.000,1.172) -- (0.000,2.310);
\draw[dbl] (0.000,1.172) -- (-0.062,0.129);
\draw[edge] (0.000,1.172) -- (0.564,-0.265);
\draw[edge] (0.000,1.172) -- (-2.520,-2.310);
\draw[dbl] (0.000,1.172) -- (-0.904,-0.740);
\draw[edge] (0.000,1.172) -- (2.520,-2.310);
\draw[edge] (-0.369,-1.359) -- (0.454,-0.987);
\draw[edge] (-0.369,-1.359) -- (0.564,-0.265);
\draw[edge] (-0.369,-1.359) -- (1.064,-1.390);
\draw[edge] (-0.369,-1.359) -- (-2.520,-2.310);
\draw[edge] (0.454,-0.987) -- (0.564,-0.265);
\draw[edge] (0.454,-0.987) -- (1.064,-1.390);
\draw[edge] (0.563,-1.923) -- (1.064,-1.390);
\draw[edge] (0.563,-1.923) -- (-2.520,-2.310);
\draw[edge] (0.563,-1.923) -- (2.520,-2.310);
\draw[edge] (0.000,2.310) -- (-2.520,-2.310);
\draw[edge] (0.000,2.310) -- (2.520,-2.310);
\draw[edge] (-0.062,0.129) -- (0.564,-0.265);
\draw[edge] (-0.062,0.129) -- (-0.904,-0.740);
\draw[edge] (0.564,-0.265) -- (1.064,-1.390);
\draw[edge] (0.564,-0.265) -- (-2.520,-2.310);
\draw[dbl] (0.564,-0.265) -- (-0.904,-0.740);
\draw[edge] (0.564,-0.265) -- (2.520,-2.310);
\draw[edge] (1.064,-1.390) -- (-2.520,-2.310);
\draw[edge] (1.064,-1.390) -- (2.520,-2.310);
\draw[edge] (-2.520,-2.310) -- (-0.904,-0.740);
\draw[dbl] (-2.520,-2.310) -- (2.520,-2.310);
\node[cyc] at (0.000,1.172) {$c_1$};
\node[cyc] at (-0.369,-1.359) {$c_2$};
\node[cyc] at (0.454,-0.987) {$c_3$};
\node[cyc] at (0.563,-1.923) {$c_4$};
\node[cyc] at (0.000,2.310) {$c_5$};
\node[clq] at (-0.062,0.129) {$k_1$};
\node[clq] at (0.564,-0.265) {$k_2$};
\node[clq] at (1.064,-1.390) {$k_3$};
\node[clq] at (-2.520,-2.310) {$k_4$};
\node[clq] at (-0.904,-0.740) {$k_5$};
\node[clq] at (2.520,-2.310) {$k_6$};
\node[ttl] at (0,-3.29) {$T_2$};
\end{scope}
\begin{scope}[xshift=11.749cm]
\draw[hedge] (0.000,2.218) -- (2.109,0.685);
\draw[hedge] (0.000,2.218) -- (-2.109,0.685);
\draw[hedge] (-1.303,-1.794) -- (-2.109,0.685);
\draw[hedge] (-1.303,-1.794) -- (1.303,-1.794);
\draw[hedge] (2.109,0.685) -- (1.303,-1.794);
\node[cyc] at (0.000,2.218) {$c_1$};
\node[cyc] at (-1.303,-1.794) {$c_2$};
\node[cyc] at (2.109,0.685) {$c_3$};
\node[cyc] at (-2.109,0.685) {$c_4$};
\node[cyc] at (1.303,-1.794) {$c_5$};
\node[clq] at (0.000,1.008) {$k_1$};
\node[clq] at (-0.873,0.504) {$k_2$};
\node[clq] at (-0.873,-0.504) {$k_3$};
\node[clq] at (-0.000,-1.008) {$k_4$};
\node[clq] at (0.873,-0.504) {$k_5$};
\node[clq] at (0.873,0.504) {$k_6$};
\node[ttl] at (0,-3.29) {$H=\overline{U}$};
\end{scope}
\end{tikzpicture}
\caption{Sulanke's graph $C_5\vee K_6$ as an exact union of two sphere triangulations $T_1$ and $T_2$ on the same 11 vertices (cycle vertices filled, clique vertices hollow).
The four edges that lie in both layers are drawn thick; the complement $H$ of the union is a 5-cycle on the cycle vertices (the complement of $C_5$) plus six isolated vertices, a triangle-free graph with $55-2\cdot 27+4=5$ edges.
The shaded triangle $c_4c_5k_5$ of $T_1$ is not a face: it separates $k_2$ from the other vertices.}
\label{fig:exactunion}
\end{figure}
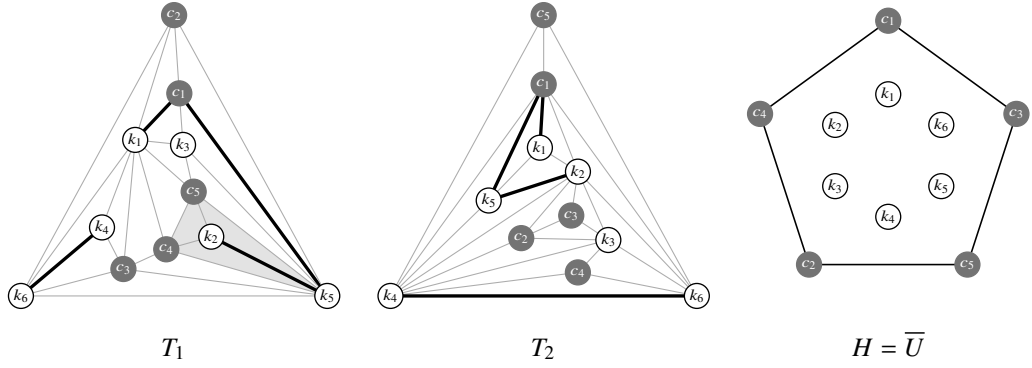
By Lemma~\ref{lem:enlarge}, Theorem~\ref{thm:nineteen} follows once no biplanar exact union on 19 vertices has independence number at most~2.
Likewise, if $G$ has a \emph{universal vertex}, that is, a vertex adjacent to all other vertices, then so has $U\supseteq G$, and Lemma~\ref{lem:apex} follows once no biplanar exact union on 18 vertices with a universal vertex has independence number at most~2.
The next lemma gives necessary conditions on the complements of such exact unions.

\begin{lemma}[the filter]
\label{lem:filter}
Let $U$ be a biplanar exact union on $19$ vertices with $\alpha(U)\le 2$ and let $H=\overline{U}$.
Then $H$ is triangle-free, has maximum degree at most $8$, has no independent set of size $9$, and $69\le\ee(H)\le 76$.
\end{lemma}

\begin{proof}
$H$ is triangle-free since $\alpha(U)\le 2$.
An independent set of $H$ is a clique of $U$, and $U$ contains no $K_9$ since $U$ is biplanar, by~(X1) and~(X3).
For a vertex $v$, the neighborhood of $v$ in $H$ is independent in $H$, hence a clique of $U$, so $\deg_H(v)\le 8$.
Let $d$ be the number of edges that lie in both triangulations.
Then $\ee(U)=2\cdot(3\cdot 19-6)-d=102-d$ and $\ee(H)=171-\ee(U)=69+d$; the degree bound gives $\ee(H)\le 19\cdot 8/2=76$.
\end{proof}

The same counting excludes orders above~19 and gives the filter for Lemma~\ref{lem:apex}, the \emph{universal-vertex case}.

\begin{lemma}[orders and the universal-vertex case]
\label{lem:orders}
The following hold.
\begin{enumerate}
\item A biplanar graph with independence number at most $2$ has at most $19$ vertices.
\item Let $U$ be a biplanar exact union on $18$ vertices with $\alpha(U)\le 2$ and a vertex $u$ adjacent to all other vertices, and let $H=\overline{U}$.
Then $u$ is isolated in $H$, $H$ is triangle-free, $\deg_H(v)\le 7$ for every vertex $v\ne u$, $H$ has no independent set of size $9$, and $57\le\ee(H)\le 59$.
\end{enumerate}
\end{lemma}

\begin{proof}
For the first part, let $G$ be biplanar on $n\ge 4$ vertices with $\alpha(G)\le 2$ and $H=\overline{G}$; for~$n\le 3$ there is nothing to prove.
As above, $H$ has maximum degree at most $8$, so $\ee(H)\le 4n$, while $\ee(G)\le 6n-12$ by~(X2) gives $\ee(H)\ge\binom{n}{2}-6n+12$.
Hence $n^2-21n+24\le 0$, which fails for~$n\ge 20$, since the left side is $4$ at $n=20$ and increases from there.
For the second part, $u$ is isolated in $H$ because it is adjacent to all other vertices in $U$, and $H$ is triangle-free without an independent set of size $9$ as in Lemma~\ref{lem:filter}.
The neighborhood of a vertex $v\ne u$ in $H$ together with $u$ is independent in $H$, hence a clique of $U$, so $\deg_H(v)\le 7$ and $\ee(H)\le\lfloor 17\cdot 7/2\rfloor=59$.
Each triangulation has $3\cdot 18-6=48$ edges, so $\ee(U)\le 96$ and $\ee(H)\ge\binom{18}{2}-96=57$.
\end{proof}

\section{Enumeration of the candidates}
\label{sec:enum}

A SAT solver decides whether a propositional formula in conjunctive normal form (\cnf, a conjunction of clauses, each a disjunction of literals) has a satisfying assignment, called a \emph{model}.
We encode the conditions of Lemma~\ref{lem:filter} as such a formula over the $171$ edge variables of a graph on 19 vertices and auxiliary variables: for each triple of vertices a clause forbidding all three edges, for each 9-subset a clause requiring one of its edges, and sequential counters (a standard \cnf encoding of cardinality bounds with auxiliary variables) for the degree bounds and for the edge count.
We enumerate its models up to isomorphism with SAT modulo symmetries (SMS)~\citep{kirchwegerszeider2024,szeider2025}.
SMS runs a conflict-driven clause learning (CDCL) solver on the formula and, during the search, adds \emph{symmetry-breaking clauses}: whenever SMS detects that the partial adjacency matrix under construction cannot be extended to one that is lexicographically minimal (with respect to a fixed order of the vertex pairs) among all its relabelings, a clause that excludes the current partial assignment is learned, together with the vertex permutation that witnesses the non-minimality.
Every isomorphism class has exactly one lexicographically minimal member, and that member survives the search; since SMS aborts its minimality test after a bounded number of steps~\citep{kirchwegerszeider2024}, a class may also contribute further models.
The run took 37 minutes and produced 3278 graphs, which fall into 3271 isomorphism classes; for the universal-vertex case of Lemma~\ref{lem:orders}, with the conditions of its second part and an isolated vertex required, SMS produced 4 graphs, which are pairwise non-isomorphic.
A plain enumeration with \texttt{geng}~\citep{mckay1998} from the nauty package~\citep{nauty} of all triangle-free graphs on 19 vertices with maximum degree at most 8 and 69 to 76 edges (1,757,679 graphs), filtered by the independence condition, gave the same 3271 graphs.

Neither SMS nor \texttt{geng} is trusted.
The completeness of the list is established in \Lean by the following argument, which uses the \Lean library LeanSMS~\citep{kirchwegermanriqueszeider2026}, ported to \Lean~4.30 and extended for enumeration (included in the archive~\citep{zenodo}).
It is stated for order 19 and run verbatim for order 18, with Lemma~\ref{lem:orders}(2) in place of Lemma~\ref{lem:filter} and the four candidates of order 18 in place of the 3271 candidates.
Let $F$ be the filter formula, $S$ the conjunction of the symmetry-breaking clauses that SMS emitted (179,420 clauses for order 19 and 29,076 for order~18), and $B$ the conjunction, over all found graphs (3278 for order 19), of a clause that is false exactly when the edge variables describe that labeled graph, whatever the auxiliary variables.
\begin{enumerate}
\item The encoding $F$ is defined in \Lean, with a proof that every graph satisfying the conditions of Lemma~\ref{lem:filter} gives a model of $F$ (including values for the auxiliary variables of the counters).
\item Each clause of $S$ comes with its permutation $\pi$.
A decidable syntactic test on the clause and $\pi$ is evaluated in \Lean, and a theorem states that a clause passing the test holds for every lexicographically minimal graph, because every labeled graph that falsifies the clause is mapped by $\pi$ to a lexicographically smaller graph.
All clauses pass.
\item $F\wedge S\wedge B$ is unsatisfiable: \CaDiCaL refutes it in 43 minutes for order 19 (385,088 clauses) and in 98 seconds for order 18, and the LRAT proof, stored in the archive, is fed from that file into the streaming checker of Section~\ref{sec:exact}.
\item For each found graph, an explicit vertex permutation onto one of the 3271 candidates is checked in \Lean.
\end{enumerate}
Consequently, every graph satisfying the conditions of Lemma~\ref{lem:filter} is isomorphic to its lexicographically minimal relabeling, which satisfies $F$ by 1, satisfies $S$ by 2, hence violates $B$ by 3, hence has the edge set of one of the found graphs, hence is isomorphic to a candidate by 4.
This certifies that the list covers all graphs satisfying the filter, which is all the proof needs; that its 3271 graphs are pairwise non-isomorphic and all satisfy the filter is checked only outside \Lean.
The SMS propagator, the component that detects non-minimal partial graphs and emits the symmetry-breaking clauses, is not trusted: only the clauses it emitted enter the argument, and each of them is verified in step~2.
This is the certification scheme of earlier SMS work~\citep[Sec.~8]{kirchwegerszeider2024}, in which an external program checks the clauses and a DRAT checker checks the refutation; here both checks run in \Lean, as in LeanSMS~\citep{kirchwegermanriqueszeider2026}.
The minimality of the found graphs is not used either.

\section{The exact-union test}
\label{sec:exact}

Fix a candidate $H$ on $n\in\{18,19\}$ vertices and let $U=\overline{H}$.
We encode necessary conditions for ``$U$ is an exact union'' as a propositional formula $\Phi(U)$.
Every exact union yields a model of $\Phi(U)$; the converse is not needed, because an unsatisfiable formula already excludes an exact union.
The clauses of $\Phi(U)$ use the following variables: for each layer $i\in\{1,2\}$, that is, each of the two triangulations, and each edge $uv$ of $U$ an edge variable $e^i_{uv}$, for each layer and each triangle $uvw$ of $U$ a face variable $f^i_{uvw}$, and for each edge $uv$ of $U$ a variable $d_{uv}$ meaning that $uv$ lies in both layers.
The clauses state that a face implies its three edges; that an edge $uv$ of layer $i$ lies in exactly two faces $f^i_{uvw}$, and a non-edge in none; that every edge of $U$ lies in at least one layer; and that $d_{uv}$ is equivalent to $e^1_{uv}\wedge e^2_{uv}$.
Sequential counters state that each layer has exactly $3n-6$ edges and that exactly $2(3n-6)-\ee(U)$ edges are doubled.
Finally, every vertex has an edge in each layer, and, to break the symmetry between the layers, the first edge of $U$ lies in layer~1.
The implication from faces to edges is deliberately not an equivalence, since a triangle of a triangulation need not be a face.

A pair of sphere triangulations with union $U$ induces a satisfying assignment of $\Phi(U)$, after exchanging the two layers if necessary, but a model of $\Phi(U)$ need not come from triangulations: the local conditions do not force a layer to be connected, nor every vertex link to be a single cycle.
To eliminate such models we strengthen $\Phi(U)$ by two families of clauses that hold for every pair of sphere triangulations, added lazily as \emph{cuts}; for a list $C$ of cuts we write $\Phi_C(U)$ for $\Phi(U)$ together with $C$.
The encoding in fact has an edge variable for every pair and a face variable for every triple of vertices, and cuts may mention any of them.
A \emph{link cut} for a vertex $v$, a cyclic sequence $c_1,\dots,c_k$ of $k\ge 3$ distinct vertices, a vertex $w\notin\{c_1,\dots,c_k\}$, and a layer $i$ is the clause
\[
\neg f^i_{v c_1 c_2}\vee\neg f^i_{v c_2 c_3}\vee\dots\vee\neg f^i_{v c_k c_1}\vee\neg e^i_{vw}:
\]
if the faces around $v$ close up along the cycle, then $v$ has no further neighbor in that layer, because the link of $v$ is a single cycle, so it contains no vertex besides $c_1,\dots,c_k$.
If $v$ lies on the cycle or~$w=v$, the clause contains a literal $\neg f^i$ or $\neg e^i$ of a degenerate triple or pair, which is never a face or an edge, so the clause holds trivially; the soundness proof covers these cases.
A \emph{connectivity cut} for a nonempty proper vertex subset $X$ and a layer $i$ is the clause $\bigvee\SB e^i_{uv}\SM u\in X,\ v\notin X,\ uv\in E(U)\SE$.
A search loop outside \Lean solves $\Phi_C(U)$, starting with $C$ empty, decodes a model, adds cuts that the model violates to $C$, and repeats until $\Phi_C(U)$ is unsatisfiable; 2596 of the 3275 instances (3271 of order 19 and 4 of order 18) need no cut, and the largest number of cuts needed by an instance is 3682.

The encoding is a \Lean function that maps an instance (the edge list of $U$) and a list of cuts to a \cnf formula, and the solver input is emitted from this function.
Its soundness theorem states: if the instance is well formed, every cut in the list is well formed (a Boolean check that the cycle is a list of at least three pairwise distinct vertices, $w$ is not on it, $X$ is nonempty and proper, and all indices are in range), and $\Phi_C(U)$ is unsatisfiable, then $U$ is not an exact union.
The proof constructs from two sphere triangulations with union $U$ the assignment that makes each edge and face variable true exactly when the edge or face belongs to the layer, so that a pair outside $U$ or a triple that is not a face is false, and shows that it satisfies every clause, including every well-formed cut and the auxiliary variables of the counters; it uses only the properties listed in~(X2).

We solved the 3271 instances of order 19 and the 4 instances of order 18 with \CaDiCaL~3.0.0~\citep{cadical3,cadical} and the options \texttt{-{}-lrat -{}-no-binary -{}-no-factor}, producing LRAT proofs.
All 3275 final formulas $\Phi_C(U)$ are unsatisfiable.
Each LRAT proof is streamed through a pipe into a \Lean process that checks it against the formula produced by the \Lean encoder.
The checker is the verified LRAT checker of the \Lean standard library~\citep{boving2025}, made resumable and fed from the pipe by lrat-catcher~\citep{lratcatcher}; the result is a \Lean theorem stating that the formula is unsatisfiable, and no proof file is stored (an earlier run that stored them needed 754~GB).
The whole run took 74 CPU-hours, executed in parallel on one compute node, with at most 3.3~GB of memory per \Lean process.
The 3275 per-instance theorems are aggregated into one theorem about the list of candidates, which the final theorems consume.
Together with Lemmas~\ref{lem:enlarge} to~\ref{lem:orders} and the completeness of the enumeration, this proves Theorem~\ref{thm:nineteen} and Lemma~\ref{lem:apex}.
An independent implementation of the exact-union test, written without access to our code, reproduced the verdicts for the 3271 candidates of order 19 and, like the \texttt{geng} run of Section~\ref{sec:enum}, the candidate list.

\section{The \Lean statements}
\label{sec:lean}

The development uses Mathlib, the mathematical library of \Lean~\citep{mathlib2020}.
In the listings, \lstinline|Fin n| is the vertex set $\{0,\dots,n-1\}$, \lstinline|#f| and \lstinline|ncard| are cardinalities, and for graphs \lstinline|≤| is edge inclusion, \lstinline|⊔| edge union (not the join), \lstinline|⊤| the complete graph, and \lstinline|G.comap f| the graph induced on the image of the embedding \lstinline|f|, relabeled.
Planarity is an arbitrary predicate \lstinline|Planar| on simple graphs with vertex set \lstinline|Fin n|, and the facts~(X1) to~(X3) are the four fields of a structure, with~(X1) split into a field for subgraphs and one for induced subgraphs.

\begin{lstlisting}
structure SphereTriangulation (n : ℕ) where
  graph : SimpleGraph (Fin n)
  faces : Finset (Finset (Fin n))
  face_card : ∀ f ∈ faces, #f = 3
  face_clique : ∀ f ∈ faces, graph.IsClique (f : Set (Fin n))
  card_edges : graph.edgeSet.ncard = 3 * n - 6
  edge_two_faces : ∀ u v, graph.Adj u v →
    #{f ∈ faces | u ∈ f ∧ v ∈ f} = 2
  link_connected : ∀ v a b, graph.Adj v a → graph.Adj v b →
    (linkGraph faces v).Reachable a b
  connected : graph.Connected

variable (Planar : ∀ {n : ℕ}, SimpleGraph (Fin n) → Prop)

def Biplanar {n : ℕ} (G : SimpleGraph (Fin n)) : Prop :=
  ∃ P₁ P₂ : SimpleGraph (Fin n), Planar P₁ ∧ Planar P₂ ∧ G ≤ P₁ ⊔ P₂

structure PlanarityFacts : Prop where
  mono : ∀ {n : ℕ} {G G' : SimpleGraph (Fin n)},
    G ≤ G' → Planar G' → Planar G
  comap : ∀ {m n : ℕ} (f : Fin m ↪ Fin n) (G : SimpleGraph (Fin n)),
    Planar G → Planar (G.comap f)
  triangulation : ∀ {n : ℕ} (G : SimpleGraph (Fin n)), 4 ≤ n →
    Planar G → ∃ T : SphereTriangulation n, G ≤ T.graph ∧ Planar T.graph
  k9 : ¬ Biplanar Planar (⊤ : SimpleGraph (Fin 9))
\end{lstlisting}

Here \lstinline|linkGraph faces v| is the link of $v$ from Section~\ref{sec:reduction}, taken as a graph on all of \lstinline|Fin n|; by \lstinline|edge_two_faces| it is 2-regular on the neighbors of $v$, so \lstinline|link_connected| makes it a single cycle, which is the cycle condition of~(X2).
For $n\ge 4$, the structure \lstinline|SphereTriangulation| is the combinatorial object of Section~\ref{sec:reduction}; the field \lstinline|triangulation| supplies its planarity.
Ordinary planarity satisfies \lstinline|PlanarityFacts| by the classical facts~(X1) to~(X3), so the theorems below are not vacuous; this instantiation is not formalized.
The final theorems read as follows: \lstinline|earth_moon_19'| is Theorem~\ref{thm:nineteen}, \lstinline|earth_moon_18'| with its hypothesis \lstinline|hu| is Lemma~\ref{lem:apex} (the unprimed versions take the completeness of the enumeration as a hypothesis, the primed ones discharge it), and \lstinline|G.IndepSetFree 3| means that \lstinline|G| has no independent set of size $3$.

\noindent\begin{minipage}{\linewidth}
\begin{lstlisting}
theorem earth_moon_19' {Planar : ∀ {n : ℕ}, SimpleGraph (Fin n) → Prop}
    (hP : PlanarityFacts Planar) (G : SimpleGraph (Fin 19))
    (hB : Biplanar Planar G) : ¬ G.IndepSetFree 3

theorem earth_moon_18' {Planar : ∀ {n : ℕ}, SimpleGraph (Fin n) → Prop}
    (hP : PlanarityFacts Planar) (G : SimpleGraph (Fin 18))
    (hB : Biplanar Planar G) (u : Fin 18)
    (hu : ∀ w, w ≠ u → G.Adj u w) : ¬ G.IndepSetFree 3
\end{lstlisting}
\end{minipage}

For \lstinline|earth_moon_19'|, the command \lstinline|#print axioms|, which reports the axioms a theorem depends on, lists the standard logical axioms \lstinline|propext|, \lstinline|Classical.choice|, and \lstinline|Quot.sound|, and 6554 axioms generated by \lstinline|native_decide| (6553 for \lstinline|earth_moon_18'|), one for each Boolean check that was evaluated by compiled code: the well-formedness of each of the 3275 instances and its cuts and each LRAT check, which both theorems share through the aggregate, the checks of the symmetry-breaking clauses, and the checks on the enumeration data.
For a streamed proof, the proof enters \Lean only through an opaque constant that supplies its bytes; the soundness theorem of the checker holds for every such input, and the axiom records only that the compiled check returned true on the bytes that arrived~\citep{lratcatcher}.
The Boolean checks are \Lean functions whose soundness is proved, so what is trusted is only their evaluation by compiled code, a weaker guarantee than evaluation in the kernel, the small core of \Lean that checks proofs.

\begin{table}[t]
\centering
\small
\begin{tabular}{@{}lp{3.6in}@{}}
\toprule
Verified in \Lean & the reduction (Section~\ref{sec:reduction}); the partition arithmetic of Section~\ref{sec:main}; completeness of the filter encoding; validity of every symmetry-breaking clause; the matching of found graphs to candidates; soundness of the exact-union encoding and of every cut; every LRAT refutation; the aggregation over all 3275 instances \\
\midrule
Trusted & the \Lean kernel and, through \lstinline|native_decide|, the \Lean compiler and runtime; the streaming reader of lrat-catcher; the identity of the encoder sources against which the per-instance certificates and the final theorems were compiled (checked byte for byte, not by \Lean); the facts~(X1) to~(X3) \\
\midrule
Not trusted & \CaDiCaL, SMS, \texttt{geng}, the cut search loop, all scripts, the independent implementation \\
\bottomrule
\end{tabular}
\caption{What is verified in \Lean, what is trusted, and what is not trusted. The argument of Section~\ref{sec:main} outside \texttt{partition\_arith} is on paper.}
\label{tab:trust}
\end{table}

\paragraph{The archive.}
The Zenodo record~\citep{zenodo} contains the \Lean sources, the 3275 instances with their cuts and logs, the enumeration data and the two enumeration proofs, and all scripts.
It describes five levels of checking.
The cheapest builds the libraries and prints the axioms of the reduction (half an hour).
The next re-emits all formulas from the \Lean encoder and compares hashes, re-verifies the symmetry-breaking clauses, and checks the enumeration proofs with an independent LRAT checker (about an hour).
The remaining levels rerun the enumeration (about 1.5 hours), rerun all refutations (74 CPU-hours), and run adversarial controls that corrupt a cut, a proof, a permutation, or a symmetry-breaking clause and confirm that the corresponding check fails.

\section{Proof of Theorem~\ref{thm:main}}
\label{sec:main}

Let $G$ be a biplanar graph with $\alpha(G)\le 2$ and suppose $\chi(G)\ge 10$.
Choose an induced subgraph~$G'$ of~$G$ with $\chi(G')\ge 10$ that is minimal with respect to vertex deletion, let $n$ be its order, and let~$H=\overline{G'}$.
Then $\chi(G')=10$ and $\chi(G'-v)=9$ for every vertex $v$, $G'$ is biplanar by~(X1), $\alpha(G')\le 2$, and $H$ is triangle-free.
By Lemma~\ref{lem:orders} and Theorem~\ref{thm:nineteen}, $n\le 18$.

For a graph $F$ with $\alpha(F)\le 2$ we have $\chi(F)=|V(F)|-\match(\overline{F})$: the color classes of a coloring have size at most $2$, the classes of size $2$ form a matching of $\overline{F}$, and conversely a matching of $\overline{F}$ with $k$ edges gives a coloring with $|V(F)|-k$ colors.
Hence $\match(H)=n-10$ and $\match(H-v)=(n-1)-9=\match(H)$ for every vertex $v$ (i.e., every vertex of $H$ is missed by some maximum matching).
By the Gallai--Edmonds structure theorem~\citep{gallai1964,edmonds1965} (see~\citep[Thm.~3.2.1]{lovaszplummer1986}), the vertices missed by some maximum matching induce a subgraph whose components are factor-critical.
Here these are all vertices of $H$, so every connected component of $H$ is \emph{factor-critical}, that is, deleting any one of its vertices leaves a graph with a perfect matching.
A factor-critical graph has odd order, and a triangle-free one is either a single vertex or has at least $5$ vertices; a factor-critical component of order $m$ has a maximum matching of size $(m-1)/2$.
Let $H$ have $r$ isolated vertices and $k$ components $H_1,\dots,H_k$ of odd orders $m_1\ge\dots\ge m_k\ge 5$.
Then $\match(H)=(n-r-k)/2$, and $\match(H)=n-10$ gives $r+k=20-n\ge 2$; since $5k\le n\le 18$, we have $k\le 3$.
The graph $G'$ is the join of $K_r$ and the complements $\overline{H_1},\dots,\overline{H_k}$.

We use three bounds.
First, $G'$ contains no $K_9$ by~(X1) and~(X3), and a maximum clique of $G'$ consists of the $r$ isolated vertices and a maximum independent set of each $H_j$; as $H_j$ is triangle-free on at least $5$ vertices, $\alpha(H_j)\ge 2$, and this gives $r+2k\le 8$.
Second, for a vertex $v$ of $H_j$, the neighborhood of $v$ in $H_j$, the $r$ isolated vertices, and two non-adjacent vertices from each other component form an independent set of $H$, hence a clique of $G'$, so $\deg_{H_j}(v)+r+2(k-1)\le 8$, that is, $\deg_{H_j}(v)\le 8-r-2(k-1)$ and therefore $\ee(H_j)\le\lfloor m_j\cdot(8-r-2(k-1))/2\rfloor$.
Third, $\binom{n}{2}-\sum_{j=1}^{k}\ee(H_j)=\ee(G')\le 6n-12$.
If $k=0$, then $n=r$ and $r+k=20-n$ give $r=10$, contradicting $r+2k\le 8$.
For $k\ge 1$, put $D=8-r-2(k-1)$ and $L(n)=\binom{n}{2}-6n+12$.
The second and third bounds give
\[
L(n)\le\sum_{j=1}^{k}\ee(H_j)\le M=\sum_{j=1}^{k}\lfloor m_jD/2\rfloor.
\]
Since every $m_j$ is odd and $\sum_j m_j=n-r$, the value $M$ depends only on $r$ and $k$: it is $D(n-r)/2$ for even $D$ and $(D(n-r)-k)/2$ for odd $D$.
The constraints $n=20-r-k\le 18$, $n-r\ge 5k$, and $r+2k\le 8$ leave the following cases.
\begin{center}
\small
\setlength{\tabcolsep}{4pt}
\begin{tabular}{@{}l*{13}{r}@{}}
\toprule
$k$ & 1 & 1 & 1 & 1 & 1 & 1 & 2 & 2 & 2 & 2 & 2 & 3 & 3 \\
$r$ & 1 & 2 & 3 & 4 & 5 & 6 & 0 & 1 & 2 & 3 & 4 & 0 & 1 \\
$n$ & 18 & 17 & 16 & 15 & 14 & 13 & 18 & 17 & 16 & 15 & 14 & 17 & 16 \\
\midrule
$L(n)$ & 57 & 46 & 36 & 27 & 19 & 12 & 57 & 46 & 36 & 27 & 19 & 46 & 36 \\
$M$ & 59 & 45 & 32 & 22 & 13 & 7 & 54 & 39 & 28 & 17 & 10 & 34 & 21 \\
\bottomrule
\end{tabular}
\end{center}
Only the first case has $L(n)\le M$, and it gives $n=18$, $r=1$, $k=1$, and $m_1=17$.
This case analysis is the \Lean theorem \lstinline|partition_arith|.
In the remaining case, $G'$ has $18$ vertices and a universal vertex, which Lemma~\ref{lem:apex} excludes.
Hence $\chi(G)\le 9$.
The order bound of Theorem~\ref{thm:main} follows since a graph with $\alpha\le 2$ and $\chi\le 9$ has at most $18$ vertices; it also follows directly from Theorem~\ref{thm:nineteen} and~(X1).
The two parts of Corollary~\ref{cor:alpha} are the contrapositives of the order bound and of the coloring bound, respectively.

\section{Concluding remarks}
\label{sec:remarks}

The bounds 9 and 12 on the largest chromatic number of a biplanar graph are untouched; the route through independence number 2, which produced Sulanke's graph and the 17-vertex examples of Section~\ref{sec:intro}, is closed.
By Corollary~\ref{cor:alpha}, a 10-chromatic biplanar graph has independence number at least 3.
The smallest candidate of that kind named in the literature is $C_7[K_4]$, the 7-cycle with every vertex replaced by a $K_4$, on 28 vertices, suggested by Thomassen~\citep{gethner2018}.
It is not biplanar: the edges between adjacent copies of $K_4$ form a triangle-free graph with 112 edges, while two triangle-free planar graphs on 28 vertices have at most $2(2\cdot 28-4)=104$ edges together~\citep[Question~7.1]{trivedi2026}.
Its thickness is exactly 3.
Replacing every vertex of $C_7[K_2]$ by a $K_2$ gives $C_7[K_4]$, every subgraph of $C_7[K_2]$ on $h$ vertices has at most $3(h-1)$ edges, so $C_7[K_2]$ has arboricity at most 3, and replacing every vertex of a graph of arboricity $k$ by a $K_2$ gives a graph of thickness at most $k$~\citep[Proposition~2]{albertsonboutingethner2010}.
Gethner and Sulanke~\citep[Open Problem~1(b)]{gethnersulanke2009} ask more generally for a biplanar graph on 28 vertices with $K_4$\hy free complement.
The method of this note transfers to that question in principle, since neither the enlargement nor the exact-union encoding depends on the order, but the enumeration would be far larger, and we do not know whether it is feasible.

Two smaller questions remain open.
The first is Open Problem~2 of Gethner and Sulanke~\citep{gethnersulanke2009}: find a biplanar graph with independence ratio below $2/17$.
A biplanar graph on 18 vertices with independence number 2 would have ratio $1/9<2/17$; by Theorem~\ref{thm:main}, 18 is the only order left for an example with independence number 2, since smaller orders give ratio at least $2/17$, and the conditions of Lemma~\ref{lem:filter} adapt to order 18.
The second is the question of Kr\'al'~\citep[Problem~4]{barbados2018} whether every biplanar graph on $n$ vertices has $\alpha\ge cn$ for some constant $c>1/12$.
Corollary~\ref{cor:alpha} gives $\alpha\ge 3$ for every biplanar graph on at least 19 vertices, where degeneracy alone gives only $\lceil n/12\rceil=2$ for $19\le n\le 24$, but this does not settle the linear question.

The same combination of a \Lean-defined encoder, LRAT proofs streamed into the checker of \Lean, and an enumeration whose completeness is certified through its symmetry-breaking clauses applies to other statements of the form ``no graph with property $P$ exists,'' provided SMS can enumerate the candidates and the property has a propositional encoding with a soundness proof.

\section*{Acknowledgments}
This work was supported by the Austrian Science Fund (FWF), project 10.55776/P36688.
The author used AI assistance for implementing the computation, drafting the text, and adversarial review, and takes full responsibility for the content.

\bibliographystyle{abbrvurl}
\bibliography{references}

\begin{thebibliography}{10}

\bibitem{alalqam2012}
R.~Alalqam.
\newblock Heuristic methods applied to difficult graph theory problems.
\newblock Master's thesis, University of Colorado Denver, 2012.

\bibitem{albertsonboutingethner2010}
M.~O. Albertson, D.~L. Boutin, and E.~Gethner.
\newblock The thickness and chromatic number of {$r$}-inflated graphs.
\newblock {\em Discret. Math.}, 310(20):2725--2734, 2010.
\newblock \href {https://doi.org/10.1016/J.DISC.2010.04.019}
  {\path{doi:10.1016/J.DISC.2010.04.019}}.

\bibitem{barbados2018}
Open problems for the {B}arbados {G}raph {T}heory {W}orkshop 2018.
\newblock
  \url{https://web.math.princeton.edu/~pds/barbados18/openproblems2018.pdf},
  2018.

\bibitem{bhk1962}
J.~Battle, F.~Harary, and Y.~Kodama.
\newblock Every planar graph with nine points has a nonplanar complement.
\newblock {\em Bulletin of the American Mathematical Society}, 68(6):569--571,
  1962.
\newblock \href {https://doi.org/10.1090/S0002-9904-1962-10850-7}
  {\path{doi:10.1090/S0002-9904-1962-10850-7}}.

\bibitem{cadical}
A.~Biere, T.~Faller, K.~Fazekas, M.~Fleury, N.~Froleyks, and F.~Pollitt.
\newblock {CaDiCaL} 2.0.
\newblock In A.~Gurfinkel and V.~Ganesh, editors, {\em Computer Aided
  Verification - 36th International Conference, {CAV} 2024, Montreal, QC,
  Canada, July 24-27, 2024, Proceedings, Part {I}}, volume 14681 of {\em
  Lecture Notes in Computer Science}, pages 133--152. Springer, 2024.
\newblock \href {https://doi.org/10.1007/978-3-031-65627-9_7}
  {\path{doi:10.1007/978-3-031-65627-9_7}}.

\bibitem{biniaz2022}
A.~Biniaz.
\newblock A short proof of the non-biplanarity of {$K_9$}.
\newblock {\em J. Graph Algorithms Appl.}, 26(1):75--80, 2022.
\newblock \href {https://doi.org/10.7155/JGAA.00582}
  {\path{doi:10.7155/JGAA.00582}}.

\bibitem{boutingethnersulanke2008}
D.~L. Boutin, E.~Gethner, and T.~Sulanke.
\newblock Thickness-two graphs part one: New nine-critical graphs, permuted
  layer graphs, and {C}atlin's graphs.
\newblock {\em J. Graph Theory}, 57(3):198--214, 2008.
\newblock \href {https://doi.org/10.1002/JGT.20282}
  {\path{doi:10.1002/JGT.20282}}.

\bibitem{boving2025}
H.~B{\"{o}}ving, S.~Bhat, L.~Cicolini, A.~C. Keizer, L.~Fr{\'{e}}not,
  A.~Mohamed, L.~Stefanesco, H.~Khan, J.~Clune, C.~W. Barrett, and T.~Grosser.
\newblock Interactive bitvector reasoning using verified bit-blasting.
\newblock {\em Proc. {ACM} Program. Lang.}, 9({OOPSLA2}):3259--3285, 2025.
\newblock \href {https://doi.org/10.1145/3763167} {\path{doi:10.1145/3763167}}.

\bibitem{brakensiek2022}
J.~Brakensiek, M.~Heule, J.~Mackey, and D.~E. Narv{\'{a}}ez.
\newblock The resolution of {Keller's} conjecture.
\newblock {\em J. Autom. Reason.}, 66(3):277--300, 2022.
\newblock \href {https://doi.org/10.1007/S10817-022-09623-5}
  {\path{doi:10.1007/S10817-022-09623-5}}.

\bibitem{clune2023}
J.~Clune.
\newblock A formalized reduction of {Keller's} conjecture.
\newblock In R.~Krebbers, D.~Traytel, B.~Pientka, and S.~Zdancewic, editors,
  {\em Proceedings of the 12th {ACM} {SIGPLAN} International Conference on
  Certified Programs and Proofs, {CPP} 2023, Boston, MA, USA, January 16-17,
  2023}, pages 90--101. {ACM}, 2023.
\newblock \href {https://doi.org/10.1145/3573105.3575669}
  {\path{doi:10.1145/3573105.3575669}}.

\bibitem{cruzfilipe2017}
L.~Cruz{-}Filipe, M.~J.~H. Heule, W.~A. Hunt, Jr., M.~Kaufmann, and
  P.~Schneider{-}Kamp.
\newblock Efficient certified {RAT} verification.
\newblock In L.~de~Moura, editor, {\em Automated Deduction - {CADE} 26 - 26th
  International Conference on Automated Deduction, Gothenburg, Sweden, August
  6-11, 2017, Proceedings}, volume 10395 of {\em Lecture Notes in Computer
  Science}, pages 220--236. Springer, 2017.
\newblock \href {https://doi.org/10.1007/978-3-319-63046-5_14}
  {\path{doi:10.1007/978-3-319-63046-5_14}}.

\bibitem{demoura2021}
L.~de~Moura and S.~Ullrich.
\newblock The {Lean} 4 theorem prover and programming language.
\newblock In A.~Platzer and G.~Sutcliffe, editors, {\em Automated Deduction -
  {CADE} 28 - 28th International Conference on Automated Deduction, Virtual
  Event, July 12-15, 2021, Proceedings}, volume 12699 of {\em Lecture Notes in
  Computer Science}, pages 625--635. Springer, 2021.
\newblock \href {https://doi.org/10.1007/978-3-030-79876-5_37}
  {\path{doi:10.1007/978-3-030-79876-5_37}}.

\bibitem{dsw2016}
V.~Dujmovic, A.~Sidiropoulos, and D.~R. Wood.
\newblock Layouts of expander graphs.
\newblock {\em Chic. J. Theor. Comput. Sci.}, 2016, 2016.
\newblock URL:
  \url{http://cjtcs.cs.uchicago.edu/articles/2016/1/contents.html}.

\bibitem{edmonds1965}
J.~Edmonds.
\newblock Paths, trees, and flowers.
\newblock {\em Canadian Journal of Mathematics}, 17:449--467, 1965.
\newblock \href {https://doi.org/10.4153/CJM-1965-045-4}
  {\path{doi:10.4153/CJM-1965-045-4}}.

\bibitem{flower2017}
J.~Flower.
\newblock Heuristic search methods for discovering thickness $n$ graphs.
\newblock Master's thesis, University of Colorado Denver, 2017.

\bibitem{gallai1964}
T.~Gallai.
\newblock Maximale {S}ysteme unabh{\"{a}}ngiger {K}anten.
\newblock {\em Magyar Tud. Akad. Mat. Kutat{\'{o}} Int. K{\"{o}}zl.},
  9:401--413, 1964.

\bibitem{gallicchio2026}
J.~Gallicchio, C.~R. Codel, J.~Avigad, and M.~J.~H. Heule.
\newblock An end-to-end verification of {Keller's} conjecture.
\newblock In E.~Komendantskaya and T.~Nipkow, editors, {\em 17th International
  Conference on Interactive Theorem Proving, {ITP} 2026, Lisbon, Portugal, July
  26-29, 2026}, volume 382 of {\em LIPIcs}, pages 26:1--26:20. Schloss Dagstuhl
  - Leibniz-Zentrum f{\"{u}}r Informatik, 2026.
\newblock \href {https://doi.org/10.4230/LIPICS.ITP.2026.26}
  {\path{doi:10.4230/LIPICS.ITP.2026.26}}.

\bibitem{gardner1980}
M.~Gardner.
\newblock Mathematical {G}ames.
\newblock {\em Scientific American}, 242(2):14--21, Feb. 1980.
\newblock \href {https://doi.org/10.1038/scientificamerican0280-14}
  {\path{doi:10.1038/scientificamerican0280-14}}.

\bibitem{gethner2018}
E.~Gethner.
\newblock To the {M}oon and beyond.
\newblock In R.~Gera, T.~W. Haynes, and S.~T. Hedetniemi, editors, {\em Graph
  Theory: Favorite Conjectures and Open Problems -- 2}, Problem Books in
  Mathematics, pages 115--133. Springer, 2018.
\newblock \href {https://doi.org/10.1007/978-3-319-97686-0_11}
  {\path{doi:10.1007/978-3-319-97686-0_11}}.

\bibitem{gethnersulanke2009}
E.~Gethner and T.~Sulanke.
\newblock Thickness-two graphs part two: More new nine-critical graphs,
  independence ratio, cloned planar graphs, and singly and doubly outerplanar
  graphs.
\newblock {\em Graphs Comb.}, 25(2):197--217, 2009.
\newblock \href {https://doi.org/10.1007/S00373-008-0833-5}
  {\path{doi:10.1007/S00373-008-0833-5}}.

\bibitem{hendreywood2019}
K.~Hendrey and D.~R. Wood.
\newblock Defective and clustered choosability of sparse graphs.
\newblock {\em Comb. Probab. Comput.}, 28(5):791--810, 2019.
\newblock \href {https://doi.org/10.1017/S0963548319000063}
  {\path{doi:10.1017/S0963548319000063}}.

\bibitem{heulescheucher2024}
M.~J.~H. Heule and M.~Scheucher.
\newblock Happy ending: An empty hexagon in every set of 30 points.
\newblock In B.~Finkbeiner and L.~Kov{\'{a}}cs, editors, {\em Tools and
  Algorithms for the Construction and Analysis of Systems - 30th International
  Conference, {TACAS} 2024, Held as Part of the European Joint Conferences on
  Theory and Practice of Software, {ETAPS} 2024, Luxembourg City, Luxembourg,
  April 6-11, 2024, Proceedings, Part {I}}, volume 14570 of {\em Lecture Notes
  in Computer Science}, pages 61--80. Springer, 2024.
\newblock \href {https://doi.org/10.1007/978-3-031-57246-3_5}
  {\path{doi:10.1007/978-3-031-57246-3_5}}.

\bibitem{hutchinson1993}
J.~P. Hutchinson.
\newblock Coloring ordinary maps, maps of empires, and maps of the {M}oon.
\newblock {\em Mathematics Magazine}, 66(4):211--226, 1993.
\newblock \href {https://doi.org/10.1080/0025570X.1993.11996124}
  {\path{doi:10.1080/0025570X.1993.11996124}}.

\bibitem{jacksonringel2000}
B.~Jackson and G.~Ringel.
\newblock Variations on {R}ingel's earth-moon problem.
\newblock {\em Discret. Math.}, 211:233--242, 2000.
\newblock \href {https://doi.org/10.1016/S0012-365X(99)00278-2}
  {\path{doi:10.1016/S0012-365X(99)00278-2}}.

\bibitem{kirchwegermanriqueszeider2026}
M.~Kirchweger, P.~Manrique, and S.~Szeider.
\newblock Formally verified graph generation with {SAT} modulo symmetries and
  {Lean}.
\newblock In A.~Biere, C.~Lutz, and S.~Negri, editors, {\em Automated Reasoning
  - 13th International Joint Conference, {I\kern0pt JCAR} 2026, Lisbon,
  Portugal, July 26-29, 2026, Proceedings, Part {I}}, volume 16688 of {\em
  Lecture Notes in Computer Science}, pages 117--135. Springer, 2026.
\newblock \href {https://doi.org/10.1007/978-3-032-32589-1_8}
  {\path{doi:10.1007/978-3-032-32589-1_8}}.

\bibitem{kirchwegerscheucherszeider2023}
M.~Kirchweger, M.~Scheucher, and S.~Szeider.
\newblock {SAT}-based generation of planar graphs.
\newblock In M.~Mahajan and F.~Slivovsky, editors, {\em 26th International
  Conference on Theory and Applications of Satisfiability Testing, {SAT} 2023,
  Alghero, Italy, July 4-8, 2023}, volume 271 of {\em LIPIcs}, pages
  14:1--14:18. Schloss Dagstuhl - Leibniz-Zentrum f{\"{u}}r Informatik, 2023.
\newblock \href {https://doi.org/10.4230/LIPICS.SAT.2023.14}
  {\path{doi:10.4230/LIPICS.SAT.2023.14}}.

\bibitem{kirchwegerszeider2024}
M.~Kirchweger and S.~Szeider.
\newblock {SAT} modulo symmetries for graph generation and enumeration.
\newblock {\em {ACM} Trans. Comput. Log.}, 25(3):1--30, 2024.
\newblock \href {https://doi.org/10.1145/3670405} {\path{doi:10.1145/3670405}}.

\bibitem{lovaszplummer1986}
L.~Lov{\'{a}}sz and M.~D. Plummer.
\newblock {\em Matching Theory}, volume 121 of {\em North-Holland Mathematics
  Studies}.
\newblock North-Holland, Amsterdam, 1986.

\bibitem{mansfield1983}
A.~Mansfield.
\newblock Determining the thickness of graphs is {NP}-hard.
\newblock {\em Mathematical Proceedings of the Cambridge Philosophical
  Society}, 93(1):9--23, 1983.
\newblock \href {https://doi.org/10.1017/S030500410006028X}
  {\path{doi:10.1017/S030500410006028X}}.

\bibitem{mckay1998}
B.~D. McKay.
\newblock Isomorph-free exhaustive generation.
\newblock {\em J. Algorithms}, 26(2):306--324, 1998.
\newblock \href {https://doi.org/10.1006/JAGM.1997.0898}
  {\path{doi:10.1006/JAGM.1997.0898}}.

\bibitem{nauty}
B.~D. McKay and A.~Piperno.
\newblock Practical graph isomorphism, {II}.
\newblock {\em J. Symb. Comput.}, 60:94--112, 2014.
\newblock \href {https://doi.org/10.1016/J.JSC.2013.09.003}
  {\path{doi:10.1016/J.JSC.2013.09.003}}.

\bibitem{mutzelodenthalscharbrodt1998}
P.~Mutzel, T.~Odenthal, and M.~Scharbrodt.
\newblock The thickness of graphs: {A} survey.
\newblock {\em Graphs Comb.}, 14(1):59--73, 1998.
\newblock \href {https://doi.org/10.1007/PL00007219}
  {\path{doi:10.1007/PL00007219}}.

\bibitem{cadical3}
F.~Pollitt, M.~Fleury, K.~Fazekas, N.~Froleyks, A.~Schidler, D.~Schreiber, and
  A.~Biere.
\newblock {CaDiCaL} 3.0 (tool paper).
\newblock In A.~Ignatiev and S.~Szeider, editors, {\em 29th International
  Conference on Theory and Applications of Satisfiability Testing, {SAT} 2026,
  Lisbon, Portugal, July 20-23, 2026}, volume 377 of {\em LIPIcs}, pages
  40:1--40:14. Schloss Dagstuhl - Leibniz-Zentrum f{\"{u}}r Informatik, 2026.
\newblock \href {https://doi.org/10.4230/LIPICS.SAT.2026.40}
  {\path{doi:10.4230/LIPICS.SAT.2026.40}}.

\bibitem{ringel1959}
G.~Ringel.
\newblock {\em F{\"{a}}rbungsprobleme auf Fl{\"{a}}chen und Graphen}, volume~2
  of {\em Mathematische Monographien}.
\newblock VEB Deutscher Verlag der Wissenschaften, Berlin, 1959.

\bibitem{subercaseauxheule2023}
B.~Subercaseaux and M.~J.~H. Heule.
\newblock The packing chromatic number of the infinite square grid is 15.
\newblock In S.~Sankaranarayanan and N.~Sharygina, editors, {\em Tools and
  Algorithms for the Construction and Analysis of Systems - 29th International
  Conference, {TACAS} 2023, Held as Part of the European Joint Conferences on
  Theory and Practice of Software, {ETAPS} 2023, Paris, France, April 22-27,
  2023, Proceedings, Part {I}}, volume 13993 of {\em Lecture Notes in Computer
  Science}, pages 389--406. Springer, 2023.
\newblock \href {https://doi.org/10.1007/978-3-031-30823-9_20}
  {\path{doi:10.1007/978-3-031-30823-9_20}}.

\bibitem{subercaseaux2024}
B.~Subercaseaux, W.~Nawrocki, J.~Gallicchio, C.~R. Codel, M.~Carneiro, and
  M.~J.~H. Heule.
\newblock Formal verification of the empty hexagon number.
\newblock In Y.~Bertot, T.~Kutsia, and M.~Norrish, editors, {\em 15th
  International Conference on Interactive Theorem Proving, {ITP} 2024, Tbilisi,
  Georgia, September 9-14, 2024}, volume 309 of {\em LIPIcs}, pages
  35:1--35:19. Schloss Dagstuhl - Leibniz-Zentrum f{\"{u}}r Informatik, 2024.
\newblock \href {https://doi.org/10.4230/LIPICS.ITP.2024.35}
  {\path{doi:10.4230/LIPICS.ITP.2024.35}}.

\bibitem{szeider2025}
S.~Szeider.
\newblock {SAT} modulo symmetries: {A} survey.
\newblock In M.~Erascu and M.~Janota, editors, {\em Proceedings of the 10th
  International Workshop on Satisfiability Checking and Symbolic Computation
  (SC-Square 2025) Collocated with The 30th International Conference on
  Automated Deduction {(CADE} 2025), Stuttgart, Germany, August 2, 2025},
  volume 4116 of {\em {CEUR} Workshop Proceedings}, pages 1--11. CEUR-WS.org,
  2025.
\newblock URL: \url{https://ceur-ws.org/Vol-4116/invited1.pdf}.

\bibitem{zenodo}
S.~Szeider.
\newblock Biplanar graphs with independence number two are 9-colorable: {Lean}
  development, certificates, and reproduction instructions.
\newblock Zenodo, 2026.
\newblock \href {https://doi.org/10.5281/zenodo.22913639}
  {\path{doi:10.5281/zenodo.22913639}}.

\bibitem{lratcatcher}
S.~Szeider.
\newblock Streaming {LRAT} certificates into {Lean} theorems, 2026.
\newblock \href {https://arxiv.org/abs/2607.00815} {\path{arXiv:2607.00815}}.

\bibitem{tanheulemyreen2023}
Y.~K. Tan, M.~J.~H. Heule, and M.~O. Myreen.
\newblock Verified propagation redundancy and compositional {UNSAT} checking in
  {CakeML}.
\newblock {\em Int. J. Softw. Tools Technol. Transf.}, 25(2):167--184, 2023.
\newblock \href {https://doi.org/10.1007/S10009-022-00690-Y}
  {\path{doi:10.1007/S10009-022-00690-Y}}.

\bibitem{mathlib2020}
{The mathlib Community}.
\newblock The {Lean} mathematical library.
\newblock In J.~Blanchette and C.~Hritcu, editors, {\em Proceedings of the 9th
  {ACM} {SIGPLAN} International Conference on Certified Programs and Proofs,
  {CPP} 2020, New Orleans, LA, USA, January 20-21, 2020}, pages 367--381.
  {ACM}, 2020.
\newblock \href {https://doi.org/10.1145/3372885.3373824}
  {\path{doi:10.1145/3372885.3373824}}.

\bibitem{trivedi2026}
A.~Trivedi.
\newblock Biplanar clique blow-ups of {$C_5$} are 9-colourable.
\newblock Zenodo, August~5, 2026.
\newblock Preprint.
\newblock \href {https://doi.org/10.5281/zenodo.22074335}
  {\path{doi:10.5281/zenodo.22074335}}.

\bibitem{tutte1963}
W.~T. Tutte.
\newblock The non-biplanar character of the complete 9-graph.
\newblock {\em Canadian Mathematical Bulletin}, 6(3):319--330, 1963.
\newblock \href {https://doi.org/10.4153/CMB-1963-026-x}
  {\path{doi:10.4153/CMB-1963-026-x}}.

\bibitem{weaver2023}
R.~C. Weaver.
\newblock Utilizing graph thickness heuristics on the {E}arth-moon problem.
\newblock {\em Rose-Hulman Undergraduate Mathematics Journal}, 24(2), 2023.
\newblock Article 5.
\newblock URL: \url{https://scholar.rose-hulman.edu/rhumj/vol24/iss2/5}.

\end{thebibliography}
\end{document}